\documentclass[12pt]{article}

\usepackage{amstext}
\usepackage{amsthm}
\usepackage{amsmath}
\usepackage{amssymb}
\usepackage{latexsym}
\usepackage{amsfonts}
\usepackage{color}
\usepackage{graphicx}
\usepackage{yfonts}
\usepackage{verbatim}
\usepackage{tikz}
\usepackage{tikz-cd}  
\usepackage{yfonts}

\AtBeginDocument{%
 \def\MR#1{} 
}

\usepackage[pagebackref,hypertexnames=false, colorlinks, citecolor=black, linkcolor=black, urlcolor=blue]{hyperref}
\usepackage[backrefs]{amsrefs}

\usepackage{latexsym}
\usepackage{amssymb}
\usepackage{euscript}

\let\cal=\mathcal      

\def\mcc{M\raise.5ex\hbox{c}C}
\def\mccarthy{M\raise.5ex\hbox{c}Carthy}

\def\h{{\cal H}}

\def\M{{\cal M}}

\def\vare{\varepsilon}

\let\i=\infty

\def\={\ = \ }

\def\A{{\cal A}}

\def\F{{\cal F}}

\def\C{\mathbb C}

\def\D{\mathbb D}

\def\be{\setcounter{equation}{\value{theorem}} \begin{equation}}
\def\ee{\end{equation} \addtocounter{theorem}{1}}
\def\beq{\begin{eqnarray*}}
\def\eeq{\end{eqnarray*}}
\def\se{\setcounter{equation}{\value{theorem}}} 
\def\att{{\addtocounter{theorem}{1}}}
\def\vs{\vskip 5pt}

\def\bexer{\begin{exercise}}
\def\eexer{\end{exercise}}

\theoremstyle{plain}
\newtheorem{theorem}{Theorem}[section]
\newtheorem{lemma}[theorem]{Lemma}
\newtheorem{proposition}[theorem]{Proposition}
\newtheorem{corollary}[theorem]{Corollary}
\theoremstyle{definition}
\newtheorem{definition}[theorem]{Definition}
\newtheorem{question}[theorem]{Question}

\newtheorem{example}[theorem]{Example}

\newcommand\mn{{\mathbb M}_n}

\renewcommand\O{\Omega}

\def\mn{{\mathbb M}_n}
\def\mnd{\mn^d}
\def\mmd{{\mathbb M}_m^d}

\def\md{{\mathbb M}^{[d]}}

\newcommand{\bsbm}{\left[ \begin{smallmatrix}}
\newcommand{\esbm}{\end{smallmatrix} \right]}

\newcommand{\bbm}{ \begin{bmatrix}}
\newcommand{\ebm}{\end{bmatrix} }

\newcommand{\bpm}{\begin{pmatrix}}
\newcommand{\epm}{\end{pmatrix}}

\newcommand{\bspm}{\left(\begin{smallmatrix}}
\newcommand{\espm}{\end{smallmatrix}\right)}

\newcommand{\bsm}{\begin{smallmatrix}}
\newcommand{\esm}{\end{smallmatrix}}

\newcommand{\bal}{\begin{align*}}
\newcommand{\eal}{\end{align*}}

\newcommand\bh{B(\h)}
\newcommand{\tensor}[2]{\text{ }{\begin{smallmatrix} #1 \\ \otimes\\ #2\end{smallmatrix}}\text{  }}

\newcommand\wF{\widetilde{F}}
\newcommand\wO{\widetilde{\Omega}}
\newcommand\wsO{\widehat{\Omega}}
\newcommand\tna{{\mathcal T}_n(\A)}
\newcommand\np{{\mathbb N}^+}

\newcommand{\cdv}{{b}} 

\title{ Operator NC functions}
\author{Meric Augat
\and
John E. M\raise.5ex\hbox{c}Carthy
\thanks{Partially supported by National Science Foundation Grant  
DMS 2054199}
}
\date{}

\begin{document}
\maketitle

\begin{abstract}
	We establish a theory of NC functions on a class of von Neumann algebras with a particular direct sum property, e.g. $B(\h)$.
	In contrast to the theory's origins, we do not rely on appealing to results from the matricial case.
	We prove that the $k^{\mathrm{th}}$ directional derivative of any NC function at a scalar point is a $k$-linear homogeneous polynomial in 
	its directions.
	Consequences include the fact that NC functions defined on domains containing scalar points can be uniformly approximated by free 
	polynomials as well as realization formulas for NC functions bounded on particular sets, e.g. the non-commutative polydisk and 
	non-commutative row ball.
\end{abstract}

\section{Introduction}

Non-commutative function theory, as first proposed in the seminal work of J.L. Taylor \cites{tay72, tay73} and 
 developed for example in the monograph 
\cite{kvv14} by Kaliuzhnyi-Verbovetskyi and Vinnikov, is a matricial theory, that is a theory of functions of $d$-tuples of matrices.
Let $\mn$ denote the $n$-by-$n$ square matrices, and 
let \[
\md \ := \ \cup_{n=1}^\infty \mnd . 
\] 
A non-commutative function $f$ defined on a domain $\Omega$ in $\md$ is a function that satisfies the following two properties.

(i) The function is graded: if $x \in \mn^d$ then $f(x) \in \mn$.

(ii) It preserves intertwining: if $L : \C^m \to \C^n$ is linear, $x = (x^1, \dots, x^d) \in \mmd$
and $y = (y^1, \dots, y^d)$ are both in $\O$ and $Lx = yL$ (this means $L x^r = y^r L$ for each $1 \leq r \leq d$), then
$Lf(x) = f(y) L $.

The theory has been very successful, and can be thought of as extending free 
polynomials in $d$ variables to non-commutative holomorphic functions.
However, the negative answer to Connes's embedding conjecture \cite{Ji2020mipre} shows evaluating non-commutative polynomials on tuples of matrices is not sufficient to fully capture certain types of information, e.g. trace positivity of a free polynomial evaluated on tuples of self-adjoint contractions \cite{KlSchw08}.
Thus, there is an incentive to understand non-commutative functions applied not to matrices, but to operators on an infinite dimensional 
Hilbert space $\h$.
Accordingly, it seems natural to exploit the fact that there are (non-canonical) identifications of a matrix of operators with an 
individual operator, and so one is led to to consider functions that map elements of $B(\h)^d$ to $B(\h)$ and preserve intertwining. 

Such functions were studied in 
\cite{amip15}
and \cite{man20}. A key assumption in those papers, however, was that the function was also sequentially continuous in the strong operator topology. This assumption was needed in order to prove that the derivatives at $0$ were actually free polynomials, by invoking this property from the matricial theory and using the density of finite rank operators in the strong operator topology. The main purpose of this note is to develop a theory of non-commutative functions of operator tuples that
does not depend on the matricial theory.


For the rest of this paper, the following will be fixed.
We shall let  $\h$ be an infinite-dimensional Hilbert space. Let $\A$ be a unital sub-algebra of $B(\h)$ that is closed in the norm topology. Let $\tna$ denote the upper triangular $n$-by-$n$ matrices with entries from 
$\A$. We shall assume that $\A$ has the following direct sum property:
\be
\label{eqa1}
\forall n \geq 1, 
 \ \exists\,  U_n : \oplus_{j=1}^n \h \to \h, \ {\rm unitary,\ with\ } \\
 U_n ( \tna)  U_n^* \subseteq \A.
\ee
Examples of such an $\A$ include $\bh$; the upper triangular matrices in $\bh$ with respect to a fixed basis; and any von Neumann algebra that can be written as a tensor product of a $I_\infty$ factor with something else.

We shall let $d$ be a positive integer, and it will denote the number of variables.
For a $d$-tuple $x \in \A^d$, we shall write its coordinates with superscripts: $x = (x^1, \dots, x^d)$.  We shall toplogize $\A^d$ with the 
relative norm topology from $\bh^d$.

\begin{definition}
\label{defram1}
A set $\O \subseteq \A^d$ is called an {\em  NC  domain} if
it is open and bounded, and closed with respect to finite direct sums in the following sense:
for each $n \geq 2$ there exists a unitary  $U_n : \h^{(n)} \to \h$ so that whenever
  $x_1, \dots , x_n \in \O$,
then 
\be
\label{eqram31}
U_n \ 
\begin{bmatrix}
x_1 & 0& \cdots & 0 \\
0 & x_2 & \cdots & 0 \\
 &  & \ddots \\
0 & 0 & \cdots & x_n
\end{bmatrix}
U_n^* \ \in \ \O .
\ee
\end{definition}

\begin{example}
\label{exama1}
{\rm
The prototypical examples of NC domains are balls. The reader is welcome to assume that $\Omega$ is either 
a non-commutative polydisk, that is of 
the form
\be
\label{eqa2}
\mathcal{P} (\A) \= \{ x \in \A^d : \max_{1 \leq r \leq d} \| x^r \| < 1 \} ,
\ee
or a non-commutative row ball, that is
\be
\label{eqa3}
\mathcal{R}(\A) \=
\{
x \in \A^d : x^1 (x^1)^* + \dots + x^d (x^d)^* < 1 
\} .
\ee
More examples are given in Section \ref{secf}.
}
\end{example}

\begin{definition}
\label{defram2}
Let $\O \subseteq \A^d$ be an NC domain. A function $F: \O \to \bh$ is {\em intertwining preserving}
 if whenever  $x,y \in \O$ and  $L: \h \to \h$ is a bounded linear operator that satisfies $Lx = yL$ (i.e., $Lx^r = y^r L$ for each $r$)
then $L F(x) = F(y) L$.

We say $F$ is an 
 {\em NC  function} if it is intertwining preserving and 
  locally bounded on
$\O$.
\end{definition}

\noindent \att \textbf{Remark \thetheorem.} For any positive integer $b$ we may similarly define an NC mapping $\F:\Omega \to B(\h)^\cdv$ 
where 
$\F = (F^1,\dots, F^\cdv)$ and each $F^i:\Omega\to B(\h)$ is an NC function.
Many of our results can be reinterpreted for NC mappings with little to no overhead.

\hphantom{.}

In Section \ref{secb} we show that every NC function is Fr\'echet  holomorphic.
Our first main result is proved in Theorem \ref{thmb1}. A scalar point $a$ is a point
each of whose components is a scalar multiple of the identity.

\begin{theorem}
\label{thma1}
Suppose $\Omega$ is an NC domain containing a scalar point $a$, and $F$ is NC on $\O$.
Then for each $k$, the $k^{\rm th}$ derivative $D^kF(a) [ h_1, \dots, h_k]$ is a symmetric homogeneous free polynomial of degree $k$ in $h_1, \dots, h_k$.
\end{theorem}

We derive several consequences of this result.
In Theorem \ref{thmd2}, we show that if $\Omega$ is a balanced NC domain, then a function $F$ on $\Omega$ 
is an NC function if and only if it can be uniformly approximated by free polynomials on every finite set.
In Theorem \ref{propf1} we show that NC functions on most balanced domains are automatically sequentially strong operator continuous. This allows us to prove that every non-commutative function on the non-commutative matrix polydisk (resp. row ball) has a unique extension to an NC function on $\mathcal{P}(\bh)$
(resp. $\mathcal{R}(\bh)$).

Similarity preserving maps of matrices were studied by C. Procesi \cite{pro76}, who showed they were all
trace polynomials. In the matricial case, this can be used to prove the analogue of Theorem \ref{thma1}
\cite{ks17}. 
In the infinite dimensional case, we cannot use this theory, which makes the proof of Theorem \ref{thma1}
more complicated. However, we can then use the theorem to prove that the only intertwining preserving
bounded
$k$-linear maps are the obvious ones, the free polynomials. In Theorem \ref{thmb2} we prove:
\begin{theorem}
Let $\O$ be an NC domain.
Let $\Lambda : [\O]^k \to \bh$ be NC and $k$-linear.
 Then  $\Lambda$ is a homogeneous free polynomial of degree $k$.
\end{theorem}

\section{Preliminaries}
\label{secb}

Throughout this section, we assume that
 $\Omega$ is an NC  domain 
  in $\A^d$, 
and 
$F: \O \to \bh$ is an NC function.   Let $\np$ denote the positive integers.

For each $n \in \np$, define the unitary and similarity envelopes by
\beq
	\wsO_n & \ := \  & \{ U^* x U\  | \ U : \h^{(n)} \to \h, {\rm \ unitary},\ x \in \Omega \} \\
	\wO_n 
	 & \ := \  & \{ S^{-1} x S \  |\  S : \h^{(n)} \to \h, {\rm \ invertible},\ x \in \Omega \}.
\eeq
Notably, for $x_1,\dots, x_n\in \Omega$, $\oplus_{j=1}^n x_j\in \wsO_n$.
We can extend $F$ to $\wO = \cup_{n=1}^\i \wO_n$ by
\be
\label{eqb1}
\widetilde{F} ( \tilde{x}) \ = \ S F(x) S^{-1},
\ee
where $\tilde{x} = S^{-1}xS$ for some $x\in \Omega$.
 
 It is straightforward to prove the following from the intertwining preserving property of $F$.
 Nevertheless, we include a proof to showcase the simplicity of working with $\widetilde{F}$ in lieu of $F$.

\begin{proposition}
\label{prb2}
The function $\wF$ defined by \eqref{eqb1} is well-defined, 
and if $\tilde{x} \in \wO_m $ and
$\tilde{y} \in  \wO_m $ satisfy $ \tilde{L} \tilde{x} = \tilde{y} \tilde{L}$ for some
linear $\tilde{L} : \tensor{\h}{\C^m} \to \tensor{\h}{\C^n}$, then
$\tilde{L}  \wF ( \tilde{x}) =  \wF ( \tilde{y}) \tilde{L} $.

In particular, if $x_j \in \O$ for $1 \leq j \leq n$, then $\wF (\oplus x_j) = \oplus F(x_j)$.
\end{proposition}

\begin{proof}
	Let $\tilde{x} = S^{-1}xS$ and $\tilde{y} = T^{-1}yT$ for $x,y\in\Omega$. Define $L:\h\to\h$ by $L = T\tilde{L}S^{-1}$ and consider the 
	following 
	intetwining:
	\begin{align*}
		Lx &= T\tilde{L}S^{-1}\hat{x}  = T\tilde{L}S^{-1}S\tilde{x}S^{-1} \\
			&= T\tilde{L}\tilde{x}S^{-1} = T\tilde{y}\tilde{L}S^{-1} = TT^{-1}yT\tilde{L}S^{-1} \\
			&= yL.
	\end{align*}
	Thus, $LF(x) = F(y)L$ and consequently
	\begin{align*}
		\tilde{L}\wF(x) 
			&= \tilde{L}S^{-1}F(x)S = T^{-1}LF(x)S = T^{-1}F(y)LS = \wF(\tilde{y})T^{-1}LS \\
			&= \wF(\tilde{y})\tilde{L}.
	\end{align*}
	
	Finally, let $P_j:\h\to\h^{(n)}$ be the inclusion of $\h$ onto the $j^{\text{th}}$-coordinate of $\h^{(n)}$.
	Observe that $(\oplus_{i=1}^n x_i)P_j = P_j x_j$.
	Hence, $\wF(\oplus_{i=1}^n x_i)P_j = P_j \wF(x_i) = P_j F(x_j)$.
	The intertwining with $P_j^*$ has $x_jP_j^* = P_j^*(\oplus_{i=1}^n x_i)$.
	Thus, $P_j^*F(x_j) = P_j^*\wF(\oplus_{i=1}^n x_i)$ and combining these two intertwining shows that $\wF(\oplus_{i=1}^n x_i)$ is a 
	diagonal block operator and
	\[
		\wF(\oplus_{i=1}^n x_i) = \oplus_{i=1}^n F(x_i).
	\]
\end{proof}

For later use, let us give a sort of converse.
\begin{lemma}
\label{lemd3}
Suppose that $\O$ is an NC domain, and $F : \O \to \bh$ satisfies 
\be
\label{eqd3}
F (S^{-1} [ x \oplus y] S ) \= S^{-1} [F (x) \oplus F(y) ] S 
\ee
whenever $S: \h \to \h^{(2)}$ and $x,y, S^{-1} [ x \oplus y] S \in \O$.
Then $F$ is intertwinining preserving.
\end{lemma}

\begin{proof}
Suppose $Lx = yL$.
Let
\[
S = \bbm
1 & L \\
0 & 1 
\ebm,
\]
and \eqref{eqd3} implies that $L F(x) = F(y) L$.
\end{proof}

Recall that  $F$ is Fr\'echet holomorphic if, for every $x \in \O$, there is an open neighborhood $G$ of $0$ in
$\A^d$ so that the Taylor series
\be
\label{eqc4}
F(x+h)  \= F(x) + \sum_{k=1}^\infty D^k F(x) [h, \dots, h] 
\ee
converges uniformly for $h$ in $G$.

Using \eqref{eqa1}, it follows that
 if $x_1, \dots, x_n \in \O$, then $\exists \vare > 0$ so that
 if $y \in \tna$ and $\| y - \oplus x_j \|  < \vare$, then $U_n y U_n^* \in \O$.
The following is proved in \cite{amip15}, and, in the form stated,
in \cite{amy20}*{Section 16.1}.

\begin{proposition}
\label{prb1}
If $\Omega\subset \A^d$ is an NC domain and $F$ is an NC function on $\Omega$ then
\begin{enumerate}
\item[(i)] The function $F$ is  Fr\'echet holomorphic.

\item[(ii)] For $x \in \O$, $h \in \A$,
\[
\wF \left(
\bbm
x & h \\
0 & x \ebm
\right) 
\=
\bbm
F(x) & DF(x)[h]\\
0 & F(x) 
\ebm .
\]
\end{enumerate}
\end{proposition} 

%

We wish to prove that when $x$ is a scalar point, each derivative in \eqref{eqc4} is actually a free  polynomial in $h$. This is straightforward for the first derivative.
\begin{lemma}
\label{lemc2}
Suppose $a = (a^1, \dots, a^d)$ is a $d$-tuple of scalar matrices in $\O$.
Then $F(a)$ is a scalar, and $DF(a)[c]$ is scalar for any scalar $d$-tuple $c$.
\end{lemma}
\begin{proof}
For any $L \in \bh$, since $La = aL$, we have $LF(a) = F(a) L$.
Therefore $F(a)$ is a scalar. For all $t$ sufficiently close to $0$, $a + tc$ is in $\O$
and $F(a+tc)- F(a)$ is scalar, therefore $DF(a)[c]$ is scalar.
\end{proof}

\begin{lemma}
Suppose $a = (a^1, \dots, a^d)$ is a $d$-tuple of scalar matrices in $\O$. Then
$DF(a)[h]$ is a linear polynomial in $h$.
\label{lemc3}
\end{lemma}
\begin{proof}
First assume that 
 $h = (h_1, 0, \dots , 0)$.
Let $\vare > 0$ be such that the closed $ \max(\vare,  \vare \| h_1\|)$ ball around $a \oplus a$ is in $\wO$.
Let $J = (1, 0, \dots, 0)$ be the scalar $d$-tuple with first entry $1$, the others $0$.
As
\[
\bbm 	1 & 0 \\ 0 & h_1	\ebm
\
\bbm 	a & \vare h \\ 0 & a	\ebm
\=
\bbm 	a & \vare J \\ 0 & a	\ebm \
\bbm 	1 & 0 \\ 0 & h_1	\ebm,
\]
we get from Proposition \ref{prb2} that
\be
\label{eqc11}
DF(a)[h] = DF(a)[J]\  h_1 .
\ee
By Lemma \ref{lemc2}, $DF(a)[J]$ is a scalar, $c_1$ say, so we get
\[
DF(a) [ (h_1, 0, \dots, 0)] \= c_1 h_1 .
\]
Permuting the coordinates and using the fact that $DF(a)[h]$ is linear in $h$,
we get that for any $h$
\[
DF(a) [h] \= \sum_{r=1}^d c_r h_r
\]
for some constants $c_r$.
\end{proof}

\section{Derivatives of NC functions are free polynomials}

The derivatives are defined inductively, by
{\small
\begin{eqnarray}
\nonumber
\lefteqn{
 D^k F(x) [h_1, \dots, h_k] \= } \\
 & \lim_{\lambda \to 0} \frac{1}{\lambda} \left( D^{k-1}F(x + \lambda h_k)[h_1,\dots, h_{k-1}] - D^{k-1}F(x)[h_1, \dots, h_{k-1}] \right).
\label{eqc2}
\end{eqnarray} 
\att}
\hspace{-0.26em}The $k^{\rm th}$ derivative is $k$-linear in $h_1, \dots, h_k$.
To extend Lemma \ref{lemc3} to higher derivatives, we need to introduce some other
operators, called nc difference-differential operators in \cite{kvv14}.

$\Delta^k F( x_1, \dots, x_{k+1}) [ h_1, \dots, h_k]$ is defined to be the $(1, k+1)$ entry in the matrix
\be
\label{eqc3}
\wF \left( \bbm
			x_1 & h_1 & 0 & 0 & \dots & 0 \\ 0 & x_2 & h_2 & 0 & \dots & 0 \\
			\vdots & \vdots & \vdots & \vdots &  & \vdots \\
			0 & 0 & 0 & 0 & \dots & x_{k+1}
		\ebm  \right)
\ee
We shall show in 
Lemma \ref{lemc5} that it is $k$-linear in $[h_1, \dots, h_k]$.

The $\Delta^k$ occur when applying $\wF$ to a bidiagonal matrix.
This is proved in \cite{kvv14}*{Thm. 3.11}.
\begin{lemma}
\label{lemap181}
Let $F$ be NC.
Then
{\small
\se
\begin{eqnarray}
\label{eqap181}
\lefteqn{
\wF \left(
\bbm
x_1 & h_1 & 0 & \dots & 0 \\
0 & x_2 & h_2 & \dots & 0 \\
\vdots & \vdots & \vdots & \dots &  \vdots \\
0 & 0 & 0 & \dots & x_{k+1} 
\ebm \right) } &&\\
\label{eqap182}
&&=
\bbm F(x_1) & \Delta^1 F(x_1, x_2) [h_1]  
& \dots & \Delta^k F(x_1, \dots, x_{k+1})[h_1, \dots, h_k] \\
0 & F(x_2)  & \dots & \Delta^{k-1} F(x_2, \dots, x_{k+1})[h_2, \dots, h_k] \\
\vdots & \vdots & \dots & \vdots \\
0 & 0 & \dots & F(x_{k+1})
\ebm.
\end{eqnarray}
\att\att
}
\end{lemma}
\begin{proof}
We will prove this by induction. For $k=1$, it is the definition of $\Delta^1$.
Assume it is proved for $k-1$.
Let $I_k$ denote the $k$-by-$k$ matrix with diagonal entries the identity, and off-diagonal entries $0$.
As
{\small
\[
\bbm
x_1 & h_1 & 0 & \dots & 0 \\
0 & x_2 & h_2 & \dots & 0 \\
\vdots & \vdots & \vdots & \dots &  \vdots \\
0 & 0 & 0 & \dots & x_{k+1} 
\ebm \
\bbm
I_k \\
0 \ebm_{(k+1) \times k } \=
\bbm
I_k \\
0 \ebm_{(k+1) \times k } \
\bbm
x_1 & h_1 & 0 & \dots & 0 \\
0 & x_2 & h_2 & \dots & 0 \\
\vdots & \vdots & \vdots & \dots &  \vdots \\
0 & 0 & 0 & \dots & x_{k} 
\ebm
\]
}
we conclude that the first $k$ columns of \eqref{eqap181} agree with those
of \eqref{eqap182}.
Similarly, intertwining by $\bbm 0 & I_k \ebm$ we get that the bottom $k$ rows agree.
Finally, the $(1,(k+1))$ entry is the definition of $\Delta^k$.
\end{proof}

A key property we need is that $\Delta^k$ is $k$-linear in the directions.
In the nc case, this is proved in \cite{kvv14}*{3.5}.
\begin{lemma}
\label{lemc5}
Let $x_1, \dots, x_{k+1} \in \O$. Then $\Delta^k F(x_1, \dots, x_{k+1})[h_1, \dots, h_k]$
is $k$-linear in $h_1, \dots, h_{k}$.
\end{lemma}
\begin{proof} Let us write $\Delta^k[h_1, \dots, h_k]$ for $\Delta^k F(x_1, \dots, x_{k+1})[h_1, \dots, h_k]$.

(i) First, we show this is linear with respect to $h_1$. Homogeneity follows from observing that
\[
\bbm
x_1 & c h_1 & 0 & \dots \\
0 & x_2 & h_2 & \dots \\
\vdots & \vdots & \vdots & \dots  \\
\ebm 
\bbm
c & 0 & 0   \\
0 & 1 &    \\
\vdots & \vdots & \ddots  
\ebm 
\=
\bbm
c & 0 & 0   \\
0 & 1 &    \\
\vdots & \vdots & \ddots  
\ebm 
\bbm
x_1 & h_1 & 0 & \dots \\
0 & x_2 & h_2 & \dots \\
\vdots & \vdots & \vdots & \dots  \\
\ebm 
\]
and using the intertwining preserving property Prop. \ref{prb2}.

To show additivity, let $p \geq 1$ and $q \geq 0$ be integers.
Let $Y$ be the $(p+k+q)\times(p+k+q)$ matrix
\[
Y \= 
\bbm
x_1 & 0 & \dots && h_1 & \dots && 0 & \dots \\
0 & x_1 & \dots && 0 & \dots && 0 & \dots \\
&  & \ddots &&  & \dots && & \dots \\
0 & \dots && x_1 & h_1'  & \dots && 0 & \dots \\
&&&0 & x_2 & h_2 & \dots & 0 & \dots\\
&&&&&\ddots &&&\\
&&&&& &x_{k+1}&0&\\
&&&&& &&x_{k+1}&\\
&&&&& &&&\ddots
\ebm .
\]
Let $L$ be the $(k+1) \times (p+k+q)$ matrix
\[
L \= \bbm
1 & 0 & \dots& 0 & \dots & & &0 & \dots &0 \\
0 & 0 & \dots& 1 & 0 & \dots & &0 &&\\
0 & \dots && 0 & 1 &&  &0&&\\
0 & \dots& &  &  &\ddots&  &&&\\
0 & \dots& &  &  && 1 &0& \dots &1\\
\ebm
\]
Let $X$ be the $(k+1)\times(k+1)$ matrix
\[
X \= 
\bbm
			x_1 & h_1 & 0 & 0 & \dots & 0 \\ 0 & x_2 & h_2 & 0 & \dots & 0 \\
			\vdots & \vdots & \vdots & \vdots &  & \vdots \\
			0 & 0 & 0 & 0 & \dots & x_{k+1}
		\ebm
\]
Then
\[
LY \=
\bbm
			x_1 & \dots &  h_1 & 0 &   \dots &&0&\dots  \\
			 0 & \dots &  x_2 & h_2 & \dots &  &&  \\
			 &  & &  \ddots & &\\
			0 & \dots & 0 &   \dots & &x_{k+1} & 0& \dots & x_{k+1}
		\ebm
\= XL .
\]
Therefore the $(1,p+k+q)$ entry of $\wF(Y)$ is $\Delta^k[h_1, \dots , h_k]$.

Let $L'$ be the matrix obtained by replacing the first row of $L$ with the row that is $1$ in the $p^{\rm th}$ entry and $0$ elsewhere. Then $L'Y =  X' L'$, where $X'$ is $X$ with $h_1$ replaced by the $d$-tuple $h_1'$.
This gives that the $(p, p+k+q)$ entry of $\wF(Y)$ is $\Delta^k[h_1', \dots , h_k]$.

Now let $L''$ be the matrix that replaces the first row of $L$ with a $1$ in both the first and $p^{\rm th}$ entry, and $X''$ be $X$ with $h_1$ replaced by $h_1 + h_1'$. Then $L'' Y = X'' L''$, and we conclude that
\[
\Delta^k[h_1, \dots , h_k] + \Delta^k[h_1', \dots , h_k] \= 
\Delta^k[h_1 + h_1', \dots , h_k] .
\]
Therefore $\Delta^k$ is linear in the first entry.

(ii) To prove that $\Delta^k$ is linear in the $i^{\rm th}$ entry, for $i \geq 2$,  choose $p,q$ so that
\[
p + i -1 \= k - i + 1 + q .
\]
Then $Y$ decomposes into a $2 \times 2$ block of $(p+i-1) \times (p+i-1)$ matrices.
\[
Y \= 
\bbm A & B \\
0 & D 
\ebm .
\]
Moreover $B$ is the matrix whose bottom left-hand entry is $h_i$, and everything else is $0$.
Therefore 
\[
\wF(Y) \=
\bbm \wF(A) & \Delta^1 \wF(A, D)[B] \\
0 & \wF(D)
\ebm , 
\]
and $ \Delta^1 \wF(A, D)[B]$ is linear in $B$ (and hence in $h_i$) by part (i).
Therefore the $(1,p+k+q)$ entry of $\wF(Y)$, which we have established is
$\Delta^k[h_1, \dots, h_k]$, is linear in $h_i$, as desired.
\end{proof}

\begin{lemma}
Suppose $a = (a_1, \dots, a_{k+1})$
 is a $(k+1)$-tuple of points  in $\O$, each of which is a $d$-tuple of scalars. Then
$\Delta^kF(a_1, \dots, a_{k+1})[h_1, \dots, h_{k}]$ is a free polynomial in  $h$, homogeneous of degree
$k$.
\label{lemc6}
\end{lemma}
\begin{proof}
Let us write $\Delta^k[h_1, \dots, h_k]$ for $\Delta^kF(a_1, \dots, a_{k+1})[h_1, \dots, h_{k}]$.
By Lemma \ref{lemc5}, we know that $\Delta^k [h_1, \dots, h_{k}]$ is $k$-linear.
So we can assume that each $h_i$ is a $d$-tuple with only one non-zero entry.
Say $h_i = H_i e_{j_i}$, where $e_{j_i}$is the $d$-tuple that is $1$ in the $j_i$ slot, $0$ else, and
$H_i$ is an operator.

Claim:
\be
\label{eqc8}
\Delta^k[ H_1 e_{j_1}, H_2 e_{j_2} , \dots ] \= H_1 H_2 \dots H_k  \Delta^k[  e_{j_1},  e_{j_2} , \dots, e_{j_k} ].
\ee
This follows from the intertwining
\beq
\lefteqn{\bbm
H_1 H_2 \dots H_k & 0 & 0 & \dots  \\ 
0 &  H_2 \dots H_k & 0 &  \dots & \\
&& \ddots & \\
&&& 1
\ebm
\bbm
a_1 & e_{j_1} & 0 & \dots \\
0 & a_2 & e_{j_2} & \dots \\
&& \ddots &\\
&&&a_{k+1}\ebm} && \\
&&=
\bbm
a_1 & e_{j_1} H_1 & 0 & \dots \\
0 & a_2 & e_{j_2} H_2 & \dots \\
&& \ddots &\\
&&&a_{k+1} \ebm
\bbm
H_1 H_2 \dots H_k & 0 & 0 & \dots  \\ 
0 &  H_2 \dots H_k & 0 &  \dots & \\
&& \ddots & \\
&&& 1
\ebm
\eeq
Let 
\[
X \= 
\bbm
a_1 & e_{j_1} H_1 & 0 & \dots \\
0 & a_2 & e_{j_2} H_2 & \dots \\
&& \ddots &\\
&&&a_{k+1} \ebm
\]
As $X$ is a $d$-tuple of $(k+1)\times(k+1)$ matrices of scalars, it commutes with any
$(k+1)\times(k+1)$ matrix that has a constant operator $L$ on the diagonal.
Therefore $L$ commutes with $\Delta^k[  e_{j_1},  e_{j_2} , \dots, e_{j_k} ]$. As $L$ is arbitrary,
it follows that $\Delta^k[  e_{j_1},  e_{j_2} , \dots, e_{j_k} ]$ is a scalar. 
So from \eqref{eqc8}, we get that 
$\Delta^k[ H_1 e_{j_1}, H_2 e_{j_2} , \dots ] $ is a constant times $H_1 H_2 \dots H_k$,
and by linearity we are done.
\end{proof}

Now we relate $\Delta^k$ to $D^k$.
\begin{lemma}
\label{lemdel}
Let $F$ be NC. Then
\[
\Delta^k F(x, \dots, x) [h, \dots , h] 
\= \frac{1}{k!} D^kF(x)[h, \dots, h] \]
\end{lemma}
\begin{proof}
Let $T$ be the upper-triangular  Toeplitz matrix given by
\[
T\ = \bbm
1 & \frac{1}{\lambda} & \frac{1}{2! \lambda^2} & \dots & \frac{1}{k! \lambda^k} \\ \\

0 & 1 &  \frac{1}{\lambda} &  \dots & \frac{1}{(k-1)! \lambda^{k-1}} \\
\vdots &\vdots & \vdots &&\vdots  \\
0 & 0 & 0 & \dots & 1 
\ebm .
\]
Its inverse is
\[
T^{-1} \= \bbm
1 & \frac{-1}{\lambda} & \frac{1}{2! \lambda^2} & \dots & \frac{(-1)^k}{k! \lambda^k} \\ \\
0 & 1 &  \frac{1}{\lambda} &  \dots & \frac{(-1)^{k-1}}{(k-1)! \lambda^{k-1}} \\
\vdots &\vdots & \vdots &&\vdots  \\
0 & 0 & 0 & \dots & 1 
\ebm .
\]
We have, componentwise in $x$ and $h$, 
\[
T \ \bbm
x & 0 & 0 & \dots & 0 \\ 
0 & x+\lambda h &  0 &  \dots &  0 \\
\vdots &\vdots & \vdots &&\vdots  \\
0 & 0 & 0 & \dots & x+ k\lambda h
\ebm
\ T^{-1} 
\=
\bbm
x &  h  & 0 & \dots & 0 \\ 
0 & x + \lambda h &  h &  \dots & 0 \\
\vdots &\vdots & \vdots &&\vdots  \\
0 & 0 & 0 & \dots & x + k\lambda h 
\ebm
\]
Therefore
\[
\Delta^k F(x, x+\lambda h, \dots, x + k\lambda h)[h,h,\dots,h] 
\=
\frac{(-1)^k}{k!\lambda^k}
\sum_{j=0}^k (-1)^j
{ k \choose j} f(x + j\lambda h) .
\]
Take the limit as $\lambda \to 0$ and the right-hand side converges
to 
\[
\frac{1}{k!} D^k F (x) [h, h, \dots, h].
\]
By continuity, the left-hand side converges to $\Delta^k F (x,\dots, x)[h, \dots, h]$.
\end{proof}

Derivatives of NC functions are symmetric.
The case $k = 2$ was proved in \cite{aug21}.

\begin{proposition}
\label{propb3}
	Suppose $F$ is an NC function and $k\geq 1$ is an integer.
	If $\sigma$ is any permutation in ${\textfrak S}_k$ then
	\[
		D^k F(x)[h_1,\dots, h_k] = D^kF(x)[h_{\sigma(1)},\dots, h_{\sigma(2)}]
	\]
	for any $x$ in the domain of $F$ and for all $h_1,\dots, h_k\in \A^d$.
\end{proposition}

\begin{proof}
The case  $k=1$ is trivial, and $k=2$ was proved in \cite{aug21}.
Assume $k \geq 3$ and that 
 the result holds for $k-1$.
	If we can show that we can swap the last two entries
	\begin{equation}
		\label{eq:deriv switch k-1 and k}
		D^kF(x)[h_1,\dots, h_{k-1},h_k] = D^kF(x)[h_1,\dots, h_k,h_{k-1}]
	\end{equation}
	and also permute the first $k-1$ entries
	\begin{equation}
		\label{eq:deriv symmetric in k}
		D^kF(x)[h_1,\dots, h_{k-1},h_k] = D^kF(x)[h_{\sigma(1)},\dots, h_{\sigma(k-1)},h_k]
	\end{equation}
	then the result follows.
Set $G = D^{k-2}F$ and consider it as a function  of $x,h_1,\dots, h_{k-2}$.
Then $G$ is an NC function, and by the $k=2$ case,
	\begin{align*}
		D^2&G(x,h_1,\dots, h_{k-2})[(\ell_0,\dots,\ell_{k-2}),(\tilde{\ell}_0,\dots, \tilde{\ell}_{k-2})] \\
			&= D^2G(x,h_1,\dots, h_{k-2})[(\tilde{\ell}_0,\dots, \tilde{\ell}_{k-2}),(\ell_0,\dots,\ell_{k-2})].
	\end{align*}
	Since
	\[
		D^2G(x,h_1,\dots,h_{k_2})[h_{k-1},0,\dots,0,h_{k},0,\dots, 0] = D^kF(x)[h_1,\dots,h_k],
	\]
	we see that Equation~\eqref{eq:deriv switch k-1 and k} holds.
	The induction hypothesis says that 
	\begin{equation}
		\label{eq:deriv symmetric in k-1}
		D^{k-1}F(x)[h_1,\dots,h_{k-1}] = D^{k-1}F(x)[h_{\sigma(1)},\dots,h_{\sigma(k-1)}].
	\end{equation}
	If $G' = D^{k-1}F$ is treated as function in $x,h_1,\dots, h_{k-1}$, then applying Equation~\eqref{eq:deriv symmetric in k-1}, we have
	\begin{align*}
		D^kF(x)[h_1,\dots, h_{k-1},h_k] 
			&= DG'(x,h_1,\dots,h_{k-1})[h_k,0,\dots,0] \\
			&= DG'(x,h_{\sigma(1)},\dots, h_{\sigma(k-1)})[h_k,0,\dots,0]\\
			&= D^kF(x)[h_{\sigma(1)},\dots, h_{\sigma(k-1)}, h_k].
	\end{align*}
	Thus, both Equation~\eqref{eq:deriv switch k-1 and k} and Equation~\eqref{eq:deriv symmetric in k} hold.
	Therefore, the $k^{\text{th}}$ derivative of $F$ is symmetric in its arguments.
\end{proof}

Combining Lemma \ref{lemc6}, Lemma \ref{lemdel} and Proposition \ref{propb3}, we get
our first main result.

\begin{theorem}
\label{thmb1}
	Suppose $\O$ is an NC domain that contains a scalar point $a$ and $F$ is an NC function on $\O$.
	Then for each $k$, the $k^{\rm th}$ derivative $D^kF(a) [ h_1, \dots, h_k]$ is a  homogeneous
	polynomial of degree $k$, it is $k$-linear,  and it is  symmetric with respect to the action of ${\textfrak S}_k$.
\end{theorem}
\begin{proof}
We know that  $D^kF (a) [h_1, \dots, h_k]$ is $k$-linear, so we can assume that
each $h_i$ is a $d$-tuple with only one entry; we can write $h_i = H_i e_{j_i}$, as in the proof of Lemma \ref{lemc6}.
We want to show that
\be
\label{eqc21}
D^k F (a) [ H_1 e_{j_1}, \dots, H_k e_{j_k} ]
\ee
is a homogeneous polynomial of degree $k$ in the operators $H_1, \dots, H_k$.
Let $s_i$ be scalars for $1 \leq i \leq k$, and consider
\be
\label{eqc20}
D^k F (a) [ s_1 H_1 e_{j_1} + \dots + s_k H_k e_{j_k}, s_1 H_1 e_{j_1} + \dots + s_k H_k e_{j_k}, \dots ].
\ee
Since all the arguments are the same, by Lemma \ref{lemdel} this agrees with $k!$ times $\Delta^k$, which
by Lemma \ref{lemc6} is a homogeneous polynomial of degree $k$.
Group the terms in \eqref{eqc20} by what the commutative monomial in $s_1, \dots, s_k$ is, and
consider  the sum of the terms  in \eqref{eqc20} that are a multiple of $s_1 \dots s_k$.
These correspond to 
\be
\label{eqc22}
\sum_{\sigma \in {\textfrak S}_k} 
D^k F (a) [ H_{\sigma(1)} e_{j_\sigma(1)}, \dots, H_\sigma(k) e_{j_\sigma(k)} ] .
\ee
By Proposition \ref{propb3}, \eqref{eqc22} is just $k!$ times \eqref{eqc21}, and hence this is a 
homogeneous polynomial in $H_1, \dots, H_k$, as desired.
\end{proof}

\section{Approximating NC functions by free polynomials}
\label{secd}

The results in this section are in improvement over those  in \cite{amip15}, as they do not need the {\em a priori} assumption that the function is sequentially strong operator continuous.
Recall that a set $\O$ in a vector space is balanced if $\alpha \O \subseteq \O$ whenever $\alpha$ is a complex number
of modulus less than or equal to $1$. Importantly, $\mathcal{P}(A)$ and $\mathcal{R}(A)$ are balanced.

If $\O$ contains a scalar point $\alpha$, and $F$ is NC on $\Omega$, then $F$ is given by a convergent series of free Taylor polynomials near $\alpha$. For convenience, we assume $\alpha = 0$.

\begin{lemma}
\label{thmd1}
Let $\O$ be an NC domain containing $0$, and let $F$ be an NC function on $\Omega$. Then there is an open set $\Upsilon \subset \O$
containing $0$, and homogeneous free polynomials $p_k$  of degree $k$ so that
\be
\label{eqd1}
F(x) \= \sum_{k=0}^\infty p_k(x) \quad \forall \ x \in \Upsilon ,
\ee
and the convergence is uniform in $\Upsilon$.
\end{lemma}

\begin{proof}
By Proposition \ref{prb1}, we know that $F$ is Fr\'echet holomorphic at $0$, and by Theorem \ref{thmb1},
we know that the $k^{\rm th}$ derivative is a homogeneous polynomial $p_k$ of degree $k$.
Therefore \eqref{eqd1} holds.
\end{proof}

\begin{theorem}
\label{thmd2}
Let $\O$ be a balanced NC domain, and $F: \O \to \bh$. The following statements are equivalent.
\begin{enumerate}
	\item[(i)] The function $F$ is NC.
	
	\item[(ii)] There is a power series expansion $\sum_{k=0}^\i p_k(x)$ that converges absolutely and locally uniformly at each point $x \in 
	\O$ to
	$F(x)$, and such that each $p_k$ is a homogeneous free polynomial of degree $k$.
	
	\item[(iii)] For any triple of points in $ \O$, there is a sequence of free polynomials  that converge
	 uniformly to $F$ on a neighborhood of each point in the triple.
\end{enumerate}
\end{theorem}

\begin{proof}
$(i) \Rightarrow (ii):$ By Lemma \ref{thmd1}, $F$ is given by a power series expansion \eqref{eqd1} in a neighborhood
of $0$. We must show that this series converges absolutely on all of $\O$.

Let $x \in \O$. Since $\O$ is open and balanced, there exists $r > 1$ so that $\D(0,r) x \subseteq \O$.
Define a function $f: \D(0,r) \to \bh$ by
\[
f(\zeta) \= F ( \zeta x) .
\]
Then $f$ is holomorphic, and so norm continuous \cite{rud91}*{Thm 3.31}.
Therefore
\[ \sup \left\{ \| f (\zeta) \| \, : \, | \zeta | = \frac{1+r}{2} \right\}  \ =: \  M \ \ < \infty.
\]
By the Cauchy integral formula,
\[
\| p_k(x) \| \= \frac{1}{k!} \left\| \frac{d^k}{d \zeta^k} f (\zeta) \big\lvert_0 \right\| \ \leq \ M \left(\frac{2}{1+r}\right)^k .
\]
Therefore the power series
$\sum p_k(x)$ converges absolutely, to $f(1) = F(x)$.

Since $F$ is NC, it is bounded on some neighborhood of $x$, and by the Cauchy estimate again, the convergence
of the power series is uniform on that neighborhood.

\vs
$(ii) \Rightarrow (iii):$
Let $x_1, x_2, x_3 \in \O$.
Let $q_k = \sum_{j=0}^k p_k$. Then $q_k(x)$ converges uniformly to $F(x)$ on an open set containing  $\{ x_1, x_2, x_3 \}$.

\vs
$(iii) \Rightarrow (i):$
Since $F$ is locally uniformly approximable by free polynomials, it is locally bounded.
To see that it is also intertwining preserving, 
we shall show that it satisfies the hypotheses of Lemma \ref{lemd3}.
Let $S: \h \to \h^{(2)}$ be invertible, and assume that $x,y$ and $z = S^{-1}[x \oplus y] S$ are all in $\O$.
Let $q_k$ be a sequence of free polynomials that approximate $F$ on $\{ x,y,z \}$.
Then
\beq
F \left(  S^{-1}
\bbm x & 0 \\
0 & y \ebm
S \right)
&\=&
\lim_k q_k \left(  S^{-1}
\bbm x & 0 \\
0 & y \ebm
S \right) \\
&=&
\lim_k \left(  S^{-1}
\bbm p_k(x) & 0 \\
0 & p_k(y) \ebm
S \right) \\
&=&
 \left(  S^{-1}
\bbm F(x) & 0 \\
0 & F( y) \ebm
S \right)  .
\eeq
So by Lemma \ref{lemd3}, $F$ is intertwining preserving.
\end{proof}

The requirement that $F$ be intertwining preserving forces $F(x)$ to always lie in the double commutant of $x$.
But if $F$ is also locally bounded on a balanced domain containing $x$, we get a much stronger conclusion as a corollary
of Theorem \ref{thmd2}.

\begin{corollary}
\label{cord1}
Suppose $F$ is an NC function on a balanced NC domain $\O$. Then $F(x)$ is in the norm closed unital algebra generated
by $\{ x^1, \dots, x^d \}$.
\end{corollary}

\section{$k$-linear NC functions}

In the following theorem, we assume that $\Lambda$ is NC as a function of all $dk$ variables at once, and
is $k$-linear if they are broken up into $d$-tuples. If we had an independent proof of Theorem \ref{thmb2}, we could use it to prove Theorem \ref{thmb1} with the aid of Lemma \ref{lemc4}. Instead, we deduce it as a consequence of Theorem \ref{thmb1}.
\begin{theorem}
\label{thmb2}
Let $\O$ be an NC domain.
Let $\Lambda : [\O]^k \to \bh$ be NC and $k$-linear.
 Then  $\Lambda$ is a homogeneous free polynomial of degree $k$.
\end{theorem}

\begin{proof}
Let ${\mathfrak h} = (h_1, \dots, h_k)$ be a $k$-tuple of $d$-tuples in $\O$.
Calculating, and using $k$-linearity,  we get
\beq
D \Lambda (x) [{\mathfrak h}] &\= & \lim_{\lambda \to 0} \frac{1}{\lambda}
[ \Lambda (x + \lambda h) - \Lambda (x) ]
\\
&=& \Lambda(h_1, x_2, \dots, x_k) + \Lambda(x_1, h_2, \dots, x_k ) + \dots .
\eeq
Repeating this calculation, we get that $D^2 \Lambda (x) [ {\mathfrak h}, {\mathfrak h}]$ is $2!$ times the sum
of $\Lambda$ evaluated at every $k$-tuple that has $k-2$ entries from $(x_1, \dots, x_k)$ and $2$ entries from ${\mathfrak h}$.
Continuing, we get
\be
\label{eqf22}
D^k \Lambda(x) [ {\mathfrak h}, \dots, {\mathfrak h}] 
\=
k! \ \Lambda (h_1, \dots, h_k) .
\ee
By Theorem \ref{thmb1}, the left-hand side of \eqref{eqf22} is a homogeneous free polynomial of degree $k$,
so the right-hand side is too.
\end{proof}

It is worth singling out a special case of Theorem \ref{thmb2}.

\begin{corollary}
\label{corc1}
Let $\Lambda : [\bh]^{dk} \to \bh$ be $k$-linear, intertwining preserving, and bounded.
Then  $\Lambda$ is a homogeneous nc polynomial of degree $k$.
\end{corollary}

\begin{lemma}
\label{lemc4}
The $k^{\rm th}$ derivative $D^kF(x)[h_1,\dots,  h_k]$ is NC on $\O \times \A^{dk}$.
If $a \in \O$ is a scalar point, then $D^kF(a)[h_1, \dots , h_k]$ is NC on $\A^{dk}$.
\end{lemma}
\begin{proof}
The first assertion follows from induction, and the observation that difference quotients preserve intertwining. The second assertion follows from the fact that if $a$ is scalar,
\[
D^k F (a) [ S^{-1}  h_1 S, \dots , S^{-1} h_k S ]
\=
D^k F (S^{-1}a S) [ S^{-1}  h_1 S, \dots , S^{-1} h_k S ].
\]
\end{proof}

\section{Realization formulas}
\label{secf}

One can generalize Example \ref{exama1}.
For $\delta$ a matrix of free polynomials, let
\[
B_\delta( \A) \ = \ \{ x \in \A^d : \| \delta(x) \| < 1 \} .
\]
These sets are all NC  domains. 
If 
\[
\delta(x) \= 
\bbm x^1 & 0 & \dots & 0 \\
0 & x^2 & \dots & 0 \\
&&\ddots \\
0 & 0 & \dots & x^d 
\ebm ,
\]
then $B_\delta( \A)$ is $\mathcal{P}(\A)$ from \eqref{eqa2}. If we set
\[
\delta(x) \= ( x^1 \ x^2 \ \cdots \ x^d) ,
\]
then $B_\delta( \A)$ is $\mathcal{R}(\A)$ from \eqref{eqa3}.

The sets $B_\delta( \A)$ are closed not just under finite direct sums, but countable direct sums, in the following sense.
\begin{definition}
A family $\{ E_k \}_{k=1}^\i$ is an exhaustion of $\O$ if
\begin{enumerate}
\item
$ E_k \subseteq\  {\rm int}(E_{k+1})$ for all $k$;
\item
$\O = \bigcup_{k=1}^\infty E_k $;
\item each $ E_k$ is bounded;
\item each $ E_k$ is closed under countable
direct sums: if $x_j$ is a  sequence in $ E_k$,
then there exists a unitary $U : \h \to \h^{(\i)}$ such that
\be
\label{eqram32}
U^{-1}  \ 
\begin{bmatrix}
x_1 & 0& \cdots & \\
0 & x_2 & \cdots \\
\cdots & \cdots & \ddots
\end{bmatrix}
U \ \in \ E_k .
\ee
\end{enumerate}
\end{definition}

If we set
\[
 E_k \= \{ x \in B_\delta( \A) : \| \delta(x) \| \leq 1 - 1/k,\ {\rm and\ } \| x \| \leq k \},
\]
then $E_k$ is an exhaustion of $B_\delta(\A)$.

We have the following automatic continuity result for NC functions on balanced domains that have an exhaustion.
\begin{theorem}
\label{propf1}
Suppose $\O \subseteq \A^d$ is a balanced NC domain that has an exhaustion $(E_k)$, and $F: \O \to \bh$ is NC
and bounded on each $E_k$.
Suppose for some $k$, there is a sequence $(x_j)$ in $E_k$ that converges to $x \in E_k$ in the strong operator topology.
Then $F(x_j)$ converges to $F(x)$ in the strong operator topology.
\end{theorem}
\begin{proof}
Let $U : \h \to \h^{(\i)}$ be a unitary so that $ U^{-1} [ \oplus x_j ] U= z \in E_k$.
Let $\Pi_j : \h^\i \to \h$ be projection onto the $j^{\rm th}$ component.
Let $L_j = \Pi_j U$.
Then $L_j z = x_j L_j$. Therefore
$
F( z) \= U^{-1}[ \oplus F(x_j) ]U .
$

Let $v$ be any unit vector, and $\vare > 0$.
By Theorem \ref{thmd2}, there is a free polynomial $p$ so that
$\| p(x) - F(x) \| < \vare/3$ and $\| p(z) - F(z) \| < \vare /3$.
Therefore $\| p(x_j) - F(x_j) \| < \vare/3$ for each $j$. 

Now choose $N$ so that $j \geq N$ implies $\| [ p(x) - p(x_j) ] v \| < \vare/3$, which we can do because
multiplication is continuous on bounded sets in the strong operator topology.
Then we get for $j \geq N$ that
\[
\| [F(x) - F(x_j) ] v \| \ \leq \ \|  F(x) - p(x) \| + \| [ p(x) - p(x_j) ] v \| + \| p(x_j) - F(x_j) \| \ \leq \ \vare .
\]
\end{proof}

\begin{definition}
Let $\delta$ be an $I \times J$ matrix of free polynomials, and $F: B_\delta(\A) \to \bh$.
A realization for $F$ consists of an auxiliary Hilbert space $\M$ 
and an isometry
\be
\label{eqrc2}
  \begin{bmatrix}A&B\\C&D\end{bmatrix} \ : \C \oplus \M^{I} \to \C \oplus \M^{J}
  \ee
  such that for all $x$ in $B_\delta(\A)$
  \be
\label{eqrc3}
F(x) \=  \tensor{A}{1} + \tensor{B}{1} \tensor{1}{\delta(x)} 
\left[ 1 - \tensor{D}{1} \tensor{1}{\delta(x)} \right]^{-1} \tensor{C}{1}.
\ee
\end{definition}
We write tensors vertically in \eqref{eqrc3} just to enhance readability.

In \cite{amip15} it was shown that if $B_\delta( \bh)$ is connected and contains $0$, then every
sequentially strong operator continuous function (in the sense of Theorem \ref{propf1}) NC function
from $B_\delta(\bh)$ that is bounded by $1$ has a realization. The strong operator continuity was needed to 
pass from a realization of $B_\delta$ in the matricial case given in \cite{amfreeII} to a realization for operators.
In light of Proposition \ref{propf1}, though, this hypothesis is automatically fulfilled.
So we get:

\begin{corollary}
\label{thmf2}
Let $\delta$ be an $I \times J$ matrix of free polynomials, and $F: B_\delta(\bh) \to \bh$ satisfy
$\sup \| F (x) \| \leq 1$. Assume that $B_\delta(\bh)$ is balanced.
Then $F$ is NC if and only if it has a realization.
\end{corollary}

As another consequence, we get that every bounded non-commutative function on $B_\delta ({\mathbb M})$
(by which we mean $\{ x \in \md :  \| \delta( x) \| < 1 \}$),  has a 
unique extension to an NC function on $B_\delta(\bh)$, where we embed $\md$ into $\bh^d$ by
choosing a basis of $\h$ and 
identifying an $n$-by-$n$ matrix with the finite rank operator that is $0$ outside the first $n$-by-$n$ block.

\begin{corollary}
\label{corf3}
Assume $B_\delta(\bh)$ is balanced. Then every non-commutative bounded function $f$ on 
$B_\delta ({\mathbb M})$ has a 
unique extension to an NC function on $B_\delta(\bh)$.
\end{corollary}
\begin{proof}
Suppose $F_1$ and $F_2$ are both extensions of $f$, and let $F = F_1 - F_2$.
As $0 \in B_\delta(\bh)$ and $\delta$ is continuous, there exists $r > 0$ so that
$r \mathcal{P}(\bh) \subseteq B_\delta(\bh)$.

Let $x \in r \mathcal{P}(\bh)$. Then  there exists a sequence $(x_j)$ in
$r \mathcal{P}({\mathbb M})$ that converges to $x$ in the strong operator topology.
As $F(x_j) = 0$ for each $j$, by Theorem \ref{propf1} we get $F(x) = 0$.
Therefore $F$ vanishes on an open subset of $B_\delta(\bh)$. 
As $F$ is holomorphic, and $B_\delta(\bh)$ is connected, we conclude that $F$ is identically zero.
\end{proof}

\begin{question}
Are the previous results true if $B_\delta(\bh)$ is not balanced?
\end{question}

If one has a realization formula (Equation~\ref{eqrc3}) for $B_\delta(\A)$, then it automatically extends to 
$B_\delta(\bh)$. We do not know how different choices of algebra 
 $\A_1$ and $\A_2$  satisfying 
\eqref{eqa1} affect the set of NC functions on their balls.




\end{document}